\documentclass{article}
\usepackage{amsmath, amsthm, amssymb}
\usepackage{tikz}
\usepackage{tikz-cd}

\newtheorem{prop}{Proposition}[section]
\newtheorem{plem}[prop]{Lemma}

\theoremstyle{definition}

\title{One dimensional groups definable in the $p$-adic numbers}
\author{Juan Pablo Acosta L\'opez}

\begin{document}
\newpage
\maketitle
\begin{abstract}
A complete list of one dimensional groups definable in the $p$-adic numbers is
given, up to a finite index subgroup and a quotient by a finite subgroup.
\end{abstract}
\section{Introduction}
The main result of the following is Proposition \ref{main-teo}. There
we give a complete list of all one dimensional groups definable in
the $p$-adics, up to finite index and quotient by finite kernel.
This is similar to the list of all groups definable in
a real closed field which are connected and one dimensional obtained in
\cite{one-dimensional-real}.

The starting point is the following result.
\begin{prop}\label{mop-teo}
Suppose $T$ is a complete first order theory in a language extending the
language of rings, such that $T$ extends the theory of fields.
Suppose $T$ is algebraically bounded and dependent.
Let $G$ be a definable group. If $G$ is definably amenable
then there is an algebraic group $H$ and a type-definable subgroup
$T\subset G$ of bounded index and a type-definable group morphism
$T\to H$ with finite kernel.
\end{prop}
This is found in Theorem 2.19 of \cite{mop}. 
We just note that algebraic boundedness is seen in greater generality
than needed in \cite{algebraic-boundedness}, and it can also be extracted
from the proofs in cell decomposition \cite{cell-decomposition}.

That one dimensional groups are definably amenable comes from the following
result.
\begin{prop}\label{abelian-by-finite}
Suppose $G$ is a definable group in $Q_p$ which is one dimensional.
Then $G$ is abelian-by-finite, and so amenable.
\end{prop}
This is found in \cite{pillay-yao}. One also finds there
 a review of the definition and
some properties of dimension.
That an abelian-by-finite group is amenable can be found for example
in Theorem 449C of \cite{measure}.

A one dimensional algebraic group over a field of characteristic $0$
is the additive group, the multiplicative group, the twisted one dimensional
torus or an elliptic curve. Each of these cases is dealt separately
in Sections \ref{Qp}, \ref{Qptimes}, \ref{torus} and \ref{elliptic} 
respectively.
The general strategy is to describe the type-definable subgroups
in these concrete cases, and then by compactness extend the inverse
of the morphism in Proposition \ref{mop-teo} to a definable or
$\lor$-definable group.
\section{Logic topology}
We take $T$ a complete theory of first order logic 
and $M\vDash T$ a monster model, which for definiteness will be taken to be 
a saturated model of cardinality $\kappa$, with $\kappa$
an inaccesible cardinal with $\kappa > |T|+\aleph_0$.
A bounded cardinal is a cardinal $<\kappa$.

A definable set is a set $X\subset M^n$ defined by a formula with parameters.
A type-definable set is a set $X\subset M^n$ which is an intersection
of a bounded number of definable sets.
An $\lor$-definable set is the complement of a type-definable set, which is to
say a bounded union of definable sets.
For a set $A\subset M$ of bounded cardinality, 
an $A$-invariant set is a set $X\subset M^n$ 
such that
$X=\sigma(X)$ for all $\sigma\in\text{Aut}(M/A)$. 
Here $\text{Aut}(M/A)$ denotes the set of 
all automorphisms $\sigma$ of $M$ such that $\sigma(a)=a$ for all $a\in A$.
An invariant set $X$ 
is a set for which there exists a bounded set $A$ such that
$X$ is $A$-invariant.
This is equivalent to saying that $X$ is a bounded union of type-definable sets.

For an invariant set $X$ we will call a set $Y\subset X$ type-definable
relative to $X$ if it is the intersection of a type-definable set with $X$.
Similarly we have relatively $\lor$-definable and relatively definable sets.

Next we define the logic topology of a bounded quotient, 
this is usually defined for type-definable sets and equivalence relations
but in section \ref{Z} we give an isomorphism 
$O(a)/o(a)\cong \mathbb{R}$, where $O(a)$ is not type-definable.
We give then definitions that include this case. 

Suppose given $X$ an invariant set and $E\subset X^2$ an invariant equivalence
relation on $X$ such that $X/E$ is bounded. 
Then the logic topology on 
$X/E$ is defined as the topology given by closed sets, 
a set $C\subset X/E$ being closed iff $\pi^{-1}(C)\subset X$ is a type definable set relative to $X$. 
Here $\pi$ denotes the canonical projection $\pi:X\to X/E$.
As a bounded intersection and a finite union 
of type definable sets is type definable this gives
a topology on $X/E$.

Next we enumerate some basic properties of this definition,
the next two results are known for the logic topology on type-definable
sets and the proofs largely go through in
 this context. See \cite{logic-topology}.
I will present this proof for convenience.
\begin{prop}
\begin{enumerate}
\item The image of a type-definable set $Y\subset X$
in $X/E$ is compact.
\item If $X$ is a type-definable set 
$E\subset X^2$ is a type-definable equivalence relation
and $Y\subset X$ is a $E$-saturated invariant subset such that
for $E_Y=E\cap Y^2$
$Y/E_Y$ is bounded,
then $Y/E_Y$ is Hausdorff.
Also, the map $\pi:Y\to Y/E_Y$ has the property that the image of 
a set type-definable relative to $Y$ is closed.
\item If $G$ is a type-definable topological group and $K\subset H\subset G$
are subgroups such that $H$ is invariant, $K$ is type-definable,
 $K$ is a normal subgroup of $H$,
and such that $H/K$ is bounded, then $H/K$ is a Hausdorff topological group.
\end{enumerate}
\end{prop}
\begin{proof}
1) This is by compactness, in detail if $\{C_i\}$ is a family of closed
subsets of $X/E$ such that $\{C_i\cap \pi(Y)\}$ 
has the finite intersection property then 
$\{\pi^{-1}(C_i)\cap Y\}$ are a bounded family of type-definable sets
with the finite intersection property, and so it has non-empty intersection.

2) Take $Z$ a type-definable set. Then
$\pi^{-1}\pi(Z\cap Y)=Y\cap \{x\in X\mid \text{there exist }x'\in Z, (x,x')\in E\}$, as $Z$ and $E$ are type-definable $\pi^{-1}\pi(Y\cap Z)$ is type-definable
relative to $Y$ and $\pi(Y\cap Z)$ is closed as required.

Now to see that $Y/E_Y$ is Hausdorff take $x,y\in Y$ such that 
$\pi(x)\neq \pi(y)$. Then for $E_z=\{x'\in X\mid (z,x')\in E\}$
we have $E_x\cap E_y=\emptyset$. As $E_x$ and $E_y$ are type-definable
then there are definable sets $D_1,D_2$ such that $E_x\subset D_1$
$E_y\subset D_2$ and $D_1\cap D_2=\emptyset$. Then by the previous paragraph
there are open sets $U_1,U_2\subset X/E$ such that
$\pi^{-1}U_i\subset D_i$.

3) By 2) this space is Hausdorff. If $C$ is type-definable and $C\cap H$ 
is $\pi$-saturated, then $(C\cap H)^{-1}=C^{-1}\cap H$ is type-definable
relative to $H$ and $\pi$-saturated, so inversion is continous in $H/K$.
Similarly one obtains that left and right translations are continous.

Finally one has to see that the product is continous at the identity,
let $1\subset U\subset H/K$ be an open set, and 
let $Z\subset G$ be type-definable
such that $H\setminus Z=\pi^{-1}U$. Then $K\cap Z=\emptyset$.
As $K=K^2$ one also has $K^2\cap Z=\emptyset$. Now by compactness there is
$K\subset V$ definable such that $(V\cap G)^2\cap Z=\emptyset$.
Now by 2) there exists $1\subset V'\subset \pi(V\cap H)$ open such that
$\pi^{-1}V'\subset  V\cap H$, so $V'^2\subset U$ as required.
\end{proof}
\begin{prop}\label{compact-generation}
If $f:X\to Y$ is a map between 
a set $X$ which is a bounded union of type-definable sets $X_i$,
such that $X_i$ are definable relative to $X$,
and $Y$ is a union of type-definable sets $Y_i$ such that
$f(X_i)\subset Y_i$ and $f|_{X_i}\subset X_i\times Y_i$ is type-definable,
 then the inverse image of a set type-definable relative to  $Y$
by $f$ is type-definable relative to $X$.
So if $E\subset X^2$ and $F\subset Y^2$ are invariant equivalence relations
such that $f$ is compatible with respect to them and $X/E, Y/F$ are bounded,
then $\bar{f}:X/E\to Y/F$ is continous.
\end{prop}
The proof is straightforward and omitted.

The next result is well known
\begin{prop}\label{open-mapping-theorem}
If $f:G\to H$ is a surjective morphism of topological groups, $G$ is
$\sigma$-compact and $H$ is locally compact Hausdorff, then 
$f$ is open.
\end{prop}
\begin{prop}\label{localgroupglobal} 
$L$ is a language and $L'$ is an extension of $L$,
$G$ is a $L$-definable group and 
$O$ is set which is a countable union of
$L'$-definable sets. 
$O$ is a commutative group with $L'$-definable product and inverse 
and there is a
$L'$-definable surjective group morphism $O\to G$. $O$
contains a $L'$-type-definable group $o$ of 
bounded index such that $o$ injects in $G$, with image $L$-type-definable.
Then there exists $O'$ a disjoint union of denumerable 
$L$-definable sets which is a group
 with a
$L$-definable product and inverse, and there is 
a $L'$-definable isomorphism $O\to O'$ such that the map $O'\to G$ is $L$-definable, and the image of $o$ in $O'$ is $L$-type-definable.
\end{prop}
Before starting the proof we clarify that 
disjoint unions of infinitely many definable sets may not be a 
subset of the monster model under consideration, however when the theory
has two 0-definable elements then finite disjoint unions can always be
considered to be a definable set.
For $\lor$-definable sets or disjoint unions of definable sets,
 a definable mapping means that the restriction
to a definable set has definable image and graph. This applies to the group
morphisms and the product and inverse in $O$ and $O'$.
\begin{proof}
Let $o\subset U\subset O$ be a $L'$-definable symmetric sets
such that the restriction of $f:O\to G$ to $U$ is injective. 
Then $f(o)\subset f(U)$, is such that $f(U)$ is $L'$-definable
and $f(o)$ is $L$-type-definable. By compactness there is 
$f(o)\subset X\subset f(U)$ such that $X$ is $L$-definable.
Replacing $U$ by $f^{-1}(X)\cap U$ we may assume $f(U)$ is $L$-definable.
Now take $V$ a symmetric $L'$-definable set such that 
$o\subset V$ and $V^4\subset U$.
Similarly replacing $V$ $f^{-1}(Y)\cap V$ we may further assme 
that $f(V)$ is $L$-definable.
A bounded number of translates of $V$ cover $O$, so by compactness, a denumerable number also cover it say $O=\bigcup_{i<\omega}a_iV$.

Take $X=\bigsqcup_{i<\omega}f(V)$, with canonical injections
$j_i:f(V)\to X$. Define a map $g:X\to O$ by 
$gj_if(x)=a_ix$. This is a surjective $L'$-definable map.
Define $W_i=f(V)\setminus (gj_i)^{-1}\bigcup_{r<i}(gj_r)f(V)$.
Then $W_i$ is $L$-definable. 
We have $O'=\bigsqcup_{i<\omega} W_i$ a disjoint union of $L$-definable
sets and bijective $L'$-definable map $g:O'\to O$. We also have
that $gfj_i(W_i)$ is $L$-definable, so $fg$ is $L$-definable by translation.

Finally we have to prove that the group structure that makes $g$ a group
isomorphism is $L$-definable. Now $g(W_i)^{-1}$ is contained in 
a finite number of translates $a_jV$ so $W_i^{-1}$ is contained in a finite
number of $W_j$. Similarly $W_iW_{i'}$ is contained in a finite number 
of $W_{j}$. So it is enough to see that product and inverse are relatively
definable. We do the product as the inverse is similar, but easier.
$(x,y,z)\in W_i\times W_j\times W_r$ is such that 
$z=xy$ iff $a_rz'=a_ix'a_jy'$ for $x',y',z'\in V$ with images under $f$
$x,y,z$. In this case $a_ra_i^{-1}a_j^{-j}\in V^3$ (here we use the hypothesis
of commutativity), so $z=xy$ iff $f(a_ra_i^{-1}a_j^{-1})z=xy$, which is
visibly $L$-definable.
\end{proof}

\section{Subgroups of $Z$}\label{Z}
We denote $Z$ a saturated model of inaccesible cardinal of the theory
of the ordered abelian group $\mathbb{Z}$. 
That is, $Z$ is a monster model of Pressburger arithmetic.

Here we determine the type-definable
subgroups of $Z$.
We shall use cell decomposition in dimension one, here is the statement.
\begin{prop}
If $X\subset Z$ is a definable set then there exist
$X=S_1\cup\cdots\cup S_n$ with $S_i$ pairwise disjoint sets
of the form $S_i=(a_i,b_i)\cup (n_iZ+r_i)$ with $a_i,b_i\in Z$
or $a_i=-\infty$ or $b_i=\infty$,
and $n_i, r_i\in\mathbb{Z}$, and $(a_i,b_i)$ infinite, or
$S_i=\{a_i\}$. Such decomposition is called a cell decomposition.
If $f:X\to Z$ is a definable function then there exist a cell decomposition
of $X=S_1\cup\cdots \cup S_n$ such that $f(x)=\frac{p_i}{n_i}(x-r_i)+s_i$
for $x\in S_i$ in an infinite cell and $p_i\in\mathbb{Z},s_i\in Z$.
\end{prop}
See for example \cite{cell-decomposition-in-Z} for a proof.

Given $a\in Z$ such that $a>n$ for all $n\in \mathbb{N}$
we denote $O_Z(a)=\{b\in Z\mid \text{ there exists }n\in\mathbb{N}, |b|<na\}$.
We denote $o_Z(a)=\{b\in Z\mid \text{ for all }n\in\mathbb{N}, n|b|<a\}$.
These are subgroups of the additive group $(Z,+)$.
\begin{prop}\label{Z-standard}
With the above notation $O_Z(a)/o_Z(a)\cong (\mathbb{R},+)$ as topological groups.
\end{prop}
\begin{proof}

First we note that $o_Z(a)$ is a convex
subgroup of the ordered group $O_Z(a)$, so $O_Z(a)/o_Z(a)$ is an ordered
abelian group.

Define a function $\mathbb{Q}\to O$ by $\frac{n}{m}\mapsto\frac{n}{m}(a-r_m)$,
where $n,m\in \mathbb{Z}$, $(n,m)=1$, $m>0$, $r_m\in\mathbb{N}$ is such that
$0\leq r_m<m$ and $a\equiv r_m \mod mZ$.
A straightforward check shows that the 
 composition of this function with the canonical projection $O_Z(a)\to O_Z(a)/o_Z(a)$
is an ordered group morphism, and that this morphism extends in a unique way
to an ordered group morphism $\mathbb{R}\to O_Z(a)/o_Z(a)$.
This morphism $\mathbb{R}\to O_Z(a)/o_Z(a)$ 
is an isomorphism of topological groups.
Indeed if we 
define $F:O_Z(a)\to \mathbb{R}$ as $F(b)=\text{Sup}\{\frac{n}{m}\mid
n,m\in\mathbb{Z}, m>0, na<mb\}$, then $F$ is the inverse.
Now observe $F^{-1}(\frac{n_1}{m},\frac{n_2}{m})=\{b\in Z\mid n_1a<mb<n_2a\}$
which is definable, so by definition of the logic topology,
$\bar{F}$ is continous.
Now notice that $o_Z(a)$ is type-definable, so $O_Z(a)/o_Z(a)$ is a Hausdorff topological group. Also, $O_Z(a)$ is a denumerable union of definable sets, so 
$O_Z(a)/o_Z(a)$ is $\sigma$-compact. As $\mathbb{R}$ is a locally compact
Hausdorff topological group, $\bar{F}$ is open by Proposition \ref{open-mapping-theorem} as required.
\end{proof}
\begin{plem}\label{def-in-Z}
If $G\subset Z$ is a definable subgroup of $Z$, then it is of the form
$G=nZ$ for a $n\in\mathbb{Z}$.
\end{plem}
\begin{proof}
Let $G=(a_1+S_1)\cup\cdots\cup (a_n+S_n)$
be a cell decomposition of $G$, where
for each $i$ $S_i=0$ or $S_i=(-b_i,c_i)\cap m_iZ$ for
$b_i,c_i\in Z_{>0}$ and $m_i\in \mathbb{Z}$ or $S_i=m_iZ$ or
$S_i=(-b_i,\infty)\cap m_iZ$ or $S_i=(-\infty,c_i)\cap m_iZ$.
Note that if $S_i$ is unbounded, then $\langle a_i +S_i\rangle=a_i\mathbb{Z} +m_iZ$ so $m_iZ\subset G$.
We note that there is at least one index $r$ such that $S_r$ is unbounded,
 otherwise $G$ is bounded
and so as it is definable it has a maximum, but $G$ is a group so this cant
happen. Then $m_rZ\subset G$ and as $Z/m_rZ\cong \mathbb{Z}/m_r\mathbb{Z}$
is finite cyclic we see that any intermediate group is of the form
$nZ$ for a $m_r|n$.
\end{proof}
\begin{plem}\label{type-definable-convex-subgroups}
If $H\subset Z$ is a type-definable convex subgroup, then
$H=0$ or $H=Z$ or $H=o_Z(a)$ or $H=\cap_{i\in I} o_Z(a_i)$ where $I$
is a bounded net such that $i<j$ implies $a_j\in o_Z(a_i)$.
\end{plem}
\begin{proof}
Let $H\subset D$, with $D$ a definable set. We need to see that 
there is $a\in Z$ such that $H\subset o_Z(a)\subset D$.
By compactness there is $H\subset S\subset D$ such that $S$ is definable 
symmetric and convex. 
If $S=H$ then $H=Z$ or $H=0$ by Lemma \ref{def-in-Z}.
Assume otherwise and take $a\in S\setminus H$ positive.
As $H$ is convex $h<a$ for all $h\in H$. As $H$ is a group,
$H\subset o_Z(a)\subset [-a,a]\subset D$. As required.
\end{proof}
If $n$ is a supernatural number then denote $nZ=\cap_{m\in\mathbb{N}, m|n}mZ$.
\begin{prop}\label{type-definable-subgroups-of-Z}
If $H\subset Z$ is a type-definable set then it is of the form
$H=C\cap nZ$ for $C$ its convex hull, which is a type-definable convex
group, and $n$ a supernatural number.
\end{prop}
\begin{proof}
Take $C$ the convex hull and $H\subset D$ a definable set.
Then there are $H\subset D_2\subset D_1$ such that $D_1-D_1\subset D$, 
and $D_2+D_2\subset D_1$ .
Then there is a cell decomposition of $D_2$, $D_2=E_1\cup\cdots\cup E_n$.
Then for one of the cells, say $E_1$, $E_1\cap H$ is unbounded above in $H$.
Then $mZ\cap C\subset E_1-E_1\subset D_1$ and $H+mZ\cap C=kZ\cap C\subset D$ for
some natural numbers $m,k$ natural numbers. This is what is required.
\end{proof}

\section{Subgroups of $Q_p$}\label{Qp}
We denote $Q_p$ a saturated model of inaccesible cardinal of the theory
of the valued field $\mathbb{Q}_p$. 
Its valuation group is denoted $Z$
and its valuation $v$.
We remark that $Z$ is a saturated model of Pressburger arithmetic 
of the same cardinal as $Q_p$.
Here we determine the type-definable groups of $Q_p$

For $Q_p$ there is a cell decomposition due to Denef 
which we will state in dimension 1.
\begin{prop}
If $Q_p$ is a $p$-adically closed field and
if $X\subset Q_p$ is definable then there exists pairwise disjoint $D_i$ with
$X=D_1\cup\cdots\cup D_n$ such that
$D_i$ are of the form $\{a_i\}$ or infinite and of the form 
$\{x\in Q_p\mid \alpha_i< v(x-a_i)<\beta_i,x-a_i\in c_iQ_p^n\}$,
or of the form $\{x\in Q_p\mid \alpha_i<v(x-a_i), x-a_i\in c_iQ_p^n\}$.
For some $\alpha_i,\beta_i\in Z\cup\{-\infty,\infty\}$, $c_i\in Q_p$
and $n_i\in\mathbb{Z}_{>0}$.
\end{prop}
See \cite{cell-decomposition} for a proof.

We will also use the analytic language. This is 
a language obtained by adding to the valued field language
one function symbol for each function $\mathbb{Z}_p^n\to \mathbb{Q}_p$
given by a power series $\sum_{\alpha}a_{\alpha}\bar{x}^\alpha$
where the sum is over the multi-indices $\alpha\in \mathbb{N}^n$
and $a_{\alpha}\in\mathbb{Q}_p$ are such that $a_{\alpha}\to 0$.
This function symbol is interpreted as
 $\mathbb{Q}_p^n\to \mathbb{Q}_p$ which is the given function 
in $\mathbb{Z}_p^n$ 
and $0$ outside. 
A model of this theory will be denoted $Q_p^{an}$. In
$Q_p^{an}$ there is also a cell decomposition due to Cluckers which in dimension 1 is as follows:
\begin{prop}
If $X\subset Q_p^{an}$ is a definable set, there exist a partition into
disjoint sets $X=D_1\cup\dots\cup D_n$ such that
$D_i=\{x\in Q_p^{an}\mid \alpha_i<v(x-a_i)<\beta_i, x-a_i\in c_iQ_p^n\}$.
\end{prop}
See \cite{analytic-cell-decomposition}

\begin{plem}
Let $b\in Q_p^{\times}$ and $\alpha,\gamma\in Z$ be such that
$\alpha\leq v(b)<\gamma$. Let $n\in\mathbb{Z}$ such that $n>0$.
Denote $S=\{c\mid v(c)\geq \alpha, v(c-b)<\gamma,\text{ and there exists }
y,c=b-by^n\}$.
Then there exists a $\beta\in Z$ and $m\in\mathbb{N}$ such that
Then $D_{\beta+m}\subset S-S$ and $S\subset D_{\beta}$
\end{plem}
\begin{proof}
Denote $D_{\rho}=\{b\mid v(b)\geq \rho\}$ and $E_{\rho}=\{b\mid v(a-b)\geq \rho\}$.
Note that by Hensel's lemma $D_{v(a)+2v_p(n)+1}\subset S$.
If $\alpha=v(a)+N$ for some $N\in\mathbb{Z}$ then we are done.
Now take the set $M=v(S\setminus\{0\})\subset Z$. This is a definable
subset of $Z$, bounded below by $\alpha$, so it has a minimum $\beta$.
Now it is clear that we may replace $\alpha$ by $\beta$ and assume
$\alpha<v(a)+n$ for all $n\in\mathbb{Z}$. 
Now let $b\in S$ be such that $v(b)=\alpha$. Then by Hensel's lemma
$\{c\in Q_p\mid v(c-b)>\alpha+2n+1\}\subset S$. But then
$D_{\alpha+2n+1}\subset S-S$ which finishes the proof.
\end{proof}
\begin{plem}
Let $a\in Q_p^{\times}$ and $\gamma\in Z$ be such that
$v(a)<\gamma$. Let $n\in\mathbb{Z}$, $n>0$.
Let $S=\{b\in Q_p\mid v(a-b)<\gamma \text{ and there exists }y, b=a-ay^n$\}.
Then $Q_p=S-S$ or $S$ is as in the previous Lemma.
\end{plem}
\begin{proof}
Take $M=v(S\setminus\{0\})$ if this set is not bounded below, then the argument
of the previous proof shows $D_{\beta+2n+1}\subset S-S$ for all $\beta\in M$
so $S-S=Q_p$. If $M$ is bounded below, then we are in the situation
of the previous Lemma.
\end{proof}
\begin{plem}
Let $S\subset Q_p$ be a definable set.
Then either $S$ is finite or $S-S=Q_p$ or
there exists $\alpha\in Z$ and $n\in\mathbb{N}$ and $X\subset S $
finite such that 
$S\subset X+D_{\alpha}$ and $D_{\alpha+n}\subset S-S$.
\end{plem}
\begin{proof}
$S$ decomposes as a finite union of cells $a_i+S_i$ where $S_i$ are as
in the previous two Lemmas or are $0$. If all the $S_i$ are $0$ or
there is $i$ such that $Q_p=S_i-S_i$ then we are done.
Otherwise take $\alpha_i$ as in the previous lemmas for every $i$
such that $S_i\neq 0$. Then for $\alpha=\text{max}_i\alpha_i$ we obtain the 
result.
\end{proof}
\begin{prop}\label{type-def-subgroups-of-Qp}
If $G$ is a definable subgroup of $Q_p$ then it is of the form $D_{\alpha}$,
$0$ or $Q_p$. A type-definable subgroup of $Q_p$ is a bounded intersection
of definable subgroups of $Q_p$.
\end{prop}
\begin{proof}
Let $L$ be a type-definable subgroup of $Q_p$, and $L\subset S\subset Q_p$
a definable set. We prove that either $Q_p=S$ or $L=0$ or there exists
$\alpha\in Z$ such that $L\subset D_{\alpha}\subset S$.
Let $T$ be a symetric definable set with $L\subset T$ and $3T\subset S$.
If $T$ is finite then $L$ is a finite group, and as $Q_p$ is torsion free
,it is trivial. If $T-T=Q_p$ then $S=Q_p$ as required. So assume
by the previous Lemma that $T\subset D_{\alpha}+X$ and $D_{\alpha+n}\subset 2T$.
Now $L+D_{\alpha+n}$ is a type-definable group and satisfies 
$L+D_{\alpha+n}\subset D_{\alpha}+X$ and $L+D_{\alpha+n}\subset S$.
Now the group generated by $D_{\alpha}+X$ is of the form
$D_{\alpha+m}\oplus\oplus_i \mathbb{Z}a_i$ where $a_i$ and $a_i-a_j$
have valuation $<\alpha+n$ for all $n\in\mathbb{Z}$. By compactness 
a type-definable subgroup is then a subgroup of $D_{\alpha+m}$ and 
as $D_{\alpha+m}/D_{\alpha+n}\cong \mathbb{Z}/p^{n-m}\mathbb{Z}$ we see that every intermediate
subgroup is of the form $D_{\alpha+k}$, as required.
\end{proof}
\begin{prop}\label{generic-to-subgroup-of-Qp}
Let $G$ be an interpretable group, such that there exists 
$L\subset G$ a type-definable group of bounded index and 
a type-definable injective group morphism $\phi:L\to (Q_p,+)$.
Then there exists a definable group $H\subset G$ 
containing $L$ and an extention of $\phi$
to an injective definable group morphism $\phi:H\to Q_p$ with image
$0$, $Q_p$ or $D_{\alpha}$.
\end{prop}
\begin{proof}
By compactness there exists a definable sets $L\subset U_1\subset U_0$
such that $U_1-U_1\subset U_0$,
and a definable extension of $\phi$ to $\phi:U_0\to Q_p$,
such that $\phi$ is injective and
such that if $a,b\in U_1$ then $\phi(a-b)=\phi(a)-\phi(b)$.
Now by compactness there is a definabe group 
$\phi(L)\subset A\subset \phi(U_1)$, with $A=0, D_{\alpha}$ or $Q_p$.
Then $H=\phi^{-1}A\cap U_1$ works.
\end{proof}
\section{Subgroups of $Q_p^{\times}$}\label{Qptimes}
Now we determine the type-definable subgroups of $Q_p^{\times}$.
I shall assume at times, for simplicity, that $p\neq 2$.

For $\alpha\in Z_{>0}$ denote $U_{\alpha}=\{b\in Q_p\mid v(b-1)\geq \alpha\}$.

\begin{plem}\label{lema-1-type-def-subgroup-of-Qptimes}
Let $L\subset Q_p^{\times}$ be a type-definable subgroup,
if $v(L)$ is nontrivial, and $L\subset S$ and $S$ definable,
then there exists $n\in\mathbb{Z}_{>0}$ 
such that $U_{n}\subset S$.
\end{plem}
\begin{proof}
Let $C$ be the convex hull of $v(L)$.
Let $L\subset S$ be a definable set.
Let $L\subset T$ be a symmetric definable set such that $T^2\subset S$.
Let $T=S_0\cup\cdots\cup S_n$ be de cell decomposition of $T$.
Then there exists one of them, say $S_0$, such that $v(S_0)\cap C$ is 
unbounded below in $C$.
If $S_0=\{b\in Q_p\mid \alpha \leq v(b-a)<\beta, b-a\in cQ_p^n\}$
we distinguish two cases, so assume first that $v(a)<\tau$ for all $\tau\in C$.

If $b\in S_0\cap v^{-1}C$ then $v(b)\in C$, so $v(a)<v(b)$ and $v(b-a)=v(a)$.
Now if  $b'\in v^{-1}C$ , then
$v((b'-a)-(b-a))=v(b'-b)\geq \text{min}(v(b'), v(b))>
v(a)+2v_p(n)=v(b-a)+2v_p(n)$,
so by Hensel's lemma we get $b'-a\in cQ_p^{n}$, we see then that $b'\in S_0$
and $v^{-1}C\subset S_0$, and this case is done.

Assume now that $\tau \leq v(a)$ for one $\tau \in C$. Then by compactness
we may obtain $\tau\in C$ such that $\tau<v(a)+m$ for all $m\in\mathbb{Z}$.
Now we choose $b\in S_0\cap v^{-1}(C)$ such that $v(b)<\tau$.
Now if $d\in U_{2v_p(n)+1}$, then
$v((bd-a)-(b-a))=v(b(d-1))=v(b)+v(d-1)>v(b)+2v_p(n)=v(b-a)+2v_p(n)$
so $bd\in S_0$. Then $d\in T^2\subset S$, as required.
\end{proof}
For the next Lemma I remark that in the standard model $\mathbb{Q}_p$
and for $p\neq 2$ we have an isomorphism 
$(\mathbb{Z}_p,+)\cong U_1(\mathbb{Q}_p)$ given 
by the exponential $z\mapsto (1+p)^{z}$. This map is not likely to be
definable in $Q_p$
however it is locally analytic 
, so it is definable in $Q_p^{an}$,
in this language there is also cell decomposition so 
the Proposition \ref{type-def-subgroups-of-Qp} remains true
and produces
\begin{plem}\label{lema-2-type-def-subgroup-of-Qptimes}
Let $p\neq 2$.
If $L\subset U_1$ is a type-definable subgroup of $Q_p^{\times}$ then
$L$ is trivial or a bounded intersection of groups of the form $U_{\alpha}$.
\end{plem}
Given $a\in Q_p^{\times}$ such that $v(a)>n$ for all $n\in\mathbb{N}$,
we define a group $H=H_a$ with underlying set
$\{b\in Q_p^{\times}\mid 0\leq v(b)<v(a)\}$
and product 
$b_1\cdot_Hb_2=b_1b_2$ if $v(b_1)+v(b_2)<v(a)$ and
$b_1\cdot_H b_2=b_1b_2a^{-1}$ if $v(b_1)+v(b_2)\geq v(a)$.

We note that
$H\cong O(a)/\langle a \rangle$ where $O(a)=\{b\in Q_p^{\times}\mid \text{ there exists }n\in\mathbb{N}\text{ such that }|v(b)|\leq nv(a)\}$. 
$O(a)$ is considered as a subgroup of the multiplicative group $Q_p^{\times}$
and $\langle a \rangle$ is the group generated by $a$. The isomorphism is given
by the morphism $O(a)\to H$ that takes $b\in O$ to $ba^{-n}$ where $n$
is the unique element of $\mathbb{Z}$ such that $nv(a)\leq v(b)<(n+1)v(a)$.

\begin{prop}\label{type-definable-subgroups-of-Qptimes}
Let $p\neq 2$.
If $L\subset Q_p^{\times}$ is a type-definable subgroup then 
either 
$v(L)$ is nontrivial 
in which case for every natural number $n$,
the definable group $G_n=(Q_p^{\times})^nL$ and the convex hull $C$
of $v(L)$ satisfy
$L=\cap_{n}G_n\cap v^{-1}C$;
or there exists a root or unity $\eta\in Z_p^{\times}$ such that
such that $L$ is in direct product $L=\langle \eta\rangle L'$,
 with $L'=L\cap U_1$ trivial  or 
a bounded intersection of groups of the form
$U_{\alpha}$.
\end{prop}
\begin{proof}
Embed $Q_p\to Q_p^{*}$, where $Q_p^{*}$ is a monster 
model of $\mathbb{Q}_p$ as a valued field together with the exponentiation map $\mathbb{Z}\to \mathbb{Q}_p$,
$x\to p^x$; the map $Q_p\to Q_p^{*}$ is taken to be elementary from
$Q_p$ into the valued field reduct of $Q_p^*$.
Replacing $Q_p$ by $Q_p^*$ we may assume that there exist a map
$Z\to Q_p^{\times}$ elementary equivalent to $x\to p^x$ in $\mathbb{Q}_p$.
We define then ${\rm ac} x=xp^{-v(x)}$ the projection
$Q_p^{\times}\to Z_p^{\times}$ associated to the short exact sequence
$1\to Z_p^{\times}\to Q_p^{\times}\to Z\to 0$ and the section 
$Z\to Q_p^{\times}, x\mapsto p^x$. In the rest of the proof
we shall have definable mean definable in the valued field language,
and say ${\rm ac}$-definable to mean definable in the language including
${\rm ac}$.

Assume first that $v(L)$ is nontrivial and take $C$ and $G_n$ as in the
statement. Let $L\subset S$ be a definable set.
Let $L\subset T$, with $T$ a symmetric definable set such that 
$T^{3}\subset S$.
By the Lemma \ref{lema-1-type-def-subgroup-of-Qptimes}
we get that there exists $n$ such that $U_n \subset T$. Define
$A=(L\cap Z_p^{\times})U_n$, this is a definable group because it contains the 
group $U_n$ which is of finite index in $Z_p^{\times}$. Now take the
 $v(L)\to Z_p^{\times}/A$ the composition of the connecting
homomorphism $v(L)\to Z_p^{\times}/L\cap Z_p^{\times}$ with the canonical
projection. 
This is an ac-type-definable morphism. By compactness there exists
$r\in\mathbb{Z}$ such that $v(L)\subset C\cap rZ$ and an ac-type-definable
group morphism 
extension $C\cap rZ\to Z_p^{\times}/A$,
see Proposition \ref{type-definable-subgroups-of-Z}.
 As the codomain is finite we
see $m(C\cap rZ)=mrC=C\cap mrZ$ is a subgroup of the kernel
of $C\cap rZ\to Z_p^{\times}/A$, (here $m=\text{Card}(Z_p/A)$).
So if we denote $L_1=v^{-1}(mrZ)\cap LU_n$ 
we get that the connecting homomorphism
in the short exact sequence $1\to A\to L_1\to v(L)\cap mrZ\to 0$ is trivial,
so the sequence is split with splitting given by the restriction of the
section $Z\to Q_p^{\times}$ to $v(L)\cap mrZ$.
This is to say $x\in L_1$ iff $v(x)\in v(L)\cap mrZ$ and $xp^{-v(x)}\in A$.
By compactness there is $s\in\mathbb{Z}$ such that if 
$v(x)\in C\cap sZ$ and $xp^{-v(x)}\in A$ then $x\in T^2$.
But the set of such $x$ form the kernel of a group morphism
$v^{-1}(C)\to Z/sZ\times Z_p/A$ with finite codomain, so it contains
$v^{-1}(C)^k=(Q_p^{\times})^k\cap v^{-1}(C)$, so $S$ contains $G_k$ as required.

Assume now that $v(L)$ s trivial. Then $L\subset Z_p^{\times}$.
If $L'=L\cap U_1$ then we get a short exact sequence
$1\to L'\to L\to \pi(L)\to 1$. 
Where $\pi:Z_p^{\times}\to \mathbb{F}_p^{\times}\cong \mathbb{Z}/(p-1)\mathbb{Z}$.
Recall that above a primitive root in $\mathbb{F}_p^{\times}$
lies a primitive $p-1$th root of $Q_p$ (by Hensel's lemma),
 so the sequence 
$1\to U_1\to Z_p^{\times}\to \mathbb{F}_p^{\times}\to 1$
is split.
Then we get the connecting homomorphism 
$\pi(L)\to U_1/L'$ and from the description of $L'$ in 
Lemma \ref{lema-2-type-def-subgroup-of-Qptimes}
we see that $U_1/L'$ only has $p^n$ torsion, 
so as it has no $p-1$-torsion 
we conclude that $\pi(L)\to U_1/L'$ is trivial, so $L$ is as required.
\end{proof}
Now in the case $p=2$, we have that $U_2\cong p^2Z_p$ in $Q_p^{an}$
with an isomorphism that sends $U_{\alpha}$ to $D_{\alpha}$, and the group
$U_1/U_2$ is cyclic of order 2. In this case we get in fact a split exact
sequence $1\to U_2\to U_1\to U_1/U_2\to 1$. It is possible to say with greater
precision the effect the group $U_1/U_2$ has on a type-definable subgroup
of $Q_p$ but we shall be satisfied with ignoring the contribution of it,
in the following version of the previous Proposition which works in $p=2$.
\begin{prop}\label{type-definable-subgroups-of-Qptimes-v2}
$L\subset Q_p^{\times}$ is a type-definable subgroup then 
either 
$v(L)$ is nontrivial 
in which case for every natural number $n$,
the definable group $G_n=(Q_p^{\times})^nL$ and the convex hull $C$
of $v(L)$ satisfy
$L=\cap_{n}G_n\cap v^{-1}C$;
or $L\subset Z_p^{\times}$ and $L\cap U_2$ is a finite index subgroup of $L$
which is a bounded intersection of groups of the form $U_{\alpha}$.
\end{prop}
The last proof goes through.
\begin{prop}\label{case-Qptimes}
If $G$ is an interpretable group, and 
$L\subset G$ is a type-definable bounded index group, and
$\phi:L\to Q_p^{\times}$ is an injective type-definable group
homomorphism, then $G$ has a finite index interpretable subgroup
$H\subset G$ which is definably isomorphic to $0$ or $(Q_p^{\times})^{n}$ or
$U_{\alpha}$ or $O(a)^n/\langle a^n\rangle$. 
\end{prop}
\begin{proof}
Consider the type-definable group $\phi(L)\subset Q_p^{\times}$,
it is of the form described in the statement of 
Proposition \ref{type-definable-subgroups-of-Qptimes-v2}.

By compactness $\psi=\phi^{-1}:\phi(L)\to L$ 
extends to a type-definable injective group homomorphism
$\psi:L'\to G$ and after a restriction to a finite index subgroup
 $L'$ is either 
$ U_{\alpha}$ or $1$ or $(Q_p^{\times})^n$ or
$o(a)^n$ (here we use Lemma \ref{type-definable-convex-subgroups}),
so without loss of generality $\phi(L)=L'$. In the first three cases
$L'$ is definable of the kind required,
the image is a definable bounded index group,
 so it is finite index by compactness so we are done. 
Assume now the remaining case.

If we denote $r_m\in\mathbb{Z}$ the element such that $0\leq r_m<m$
and $v(a)-r_m\in mZ$, and $I_m=v^{-1}([\frac{1}{m}(a-r_m),\frac{1}{m}(a-r_m)])$
then $\cap_mI_m\cap (Q_p^{\times})^n=o(a)^n$, so $\psi$ extends
to an injective map $\psi_1:I_m\cap O(a)^n\to G$. 
We also may find a $m'>m$ such that $\psi_1(ab^{-1})=\psi_1(a)\psi_1(b)^{-1}$
for all $a,b\in I_{m'}\cap O(a)^n$. Remembering that
$O(a)/o(a)\cong (\mathbb{R},+)$, from which we obtain
$O(a)^n/o(a)^n\cong \mathbb{R}$, we see that the map
$\psi_1$ restricted to $I_{m'}\cap O(a)^n$ extends to a definable group
homomorphism $\psi_2:O(a)^n\to G$. The set $A=\psi_2(I_{m'}\cap O(a)^n)$
is definable and a bounded number of translates cover $G$, so a finite
number of translates cover $G$, so the group generated by it can be generated
in a finite number of steps and is definable,
indeed the number of translates of $A$
in $A^r$ stabilizes, for such an $r$ $A^r$ is the group generated by $A$.
We see then that $H=\psi_2(O(a)^n)$ is a definable subgroup of $G$ of finite
index. The kernel of $\psi_2$ is a definable subgroup $K$ of $O(a)^n$
such that $K\cap o(a)^n=1$.
The image of $K$ to $\mathbb{R}\cong O(a)^n/o(a)^n$ is a subgroup 
$f(K)$ such that $f(K)\cap [-1,1]$ is compact and $f(K)\cap[-\epsilon,\epsilon]=0$.
 The subgroups of $\mathbb{R}$
are either dense or of the form $f(K)=f(b)\mathbb{Z}$, or trivial.
$K$ can not be trivial as in this case $\psi_2$ is injective to a definable
set, which can not happen by compactness.
So  $K=\langle b\rangle$ and we are done.
\end{proof}
\section{Subgroups of the one dimensional twisted torus}\label{torus}
Given $d\in \mathbb{Q}_p\setminus \mathbb{Q}_p^2$, such that $v(d)\geq 0$,
 $G=G(d)=\{x+y\sqrt{d}\mid x^2-dy^2=1\}
\subset Q_p(\sqrt{d})^{\times}$ is what we call
a one dimensional twisted torus.
In this section we give the type-definable subgroups of $G(d)$.

This group $G(d)$ is affine, so it is a subgroup of a 
general linear group, explicitely this can be seen as follows;
an element $a\in G(d)$ produces a multiplication map
$L_a:Q_p(\sqrt{d})\to Q_p(\sqrt{d})$, and as $Q_p(\sqrt{d})$
is 2 dimensional $Q_p$-vector space this produces an injective
group morphism $G(d)\to {\rm GL}_2(Q_p)$. If one takes as 
a basis of $Q_p(\sqrt{d})$ $\{1,\sqrt{d}\}$ this map is given
by $(x,y)\mapsto\begin{bmatrix} x & dy\\ y & x \end{bmatrix}$,
the norm is precisely the determinant of this matrix.

\begin{prop}
If $L$ is a bounded index type-definable group of $(G(d),\cdot)$
for $d\in \mathbb{Q}_p^{\times}\setminus(\mathbb{Q}_p^{\times})^2$
with $v(d)\geq 0$, then $L\cap F_{2}$ is a bounded intersection of 
groups of the form $F_{\alpha}$. Where 
$F_{\alpha}=\{\begin{bmatrix} x & dy \\ y & x\end{bmatrix}\mid
v(1-x)\geq \alpha, v(y)\geq \alpha, x^2-dy^2=1\}$ and 
$F_2$ is finite index in $G(d)$.
\end{prop}
\begin{proof}
Observe that
$F_2(\mathbb{Q}_p)\subset {\rm GL}(\mathbb{Z}_p)$ has the topology
given by a $p$-valuation in the group $G(\mathbb{Q}_p)$. 
See \cite{p-adic-lie} Example 23.2. 
As this is a 1-dimensional compact
Lie group then this $p$-valuation has rank 1,
see \cite{p-adic-lie} Theorem 27.1, and (the proof of) Proposition 26.15.
Then by \cite{p-adic-lie} Proposition 26.6 this group is isomorphic to
$(p^2\mathbb{Z}_p,+)$. By \cite{p-adic-lie} Theorem 29.8 this isomorphism
is locally analytic. So this isomorphism extends to an isomorphism
$p^2Z_p\to F_2$ in $Q_p^{an}$. From the definition of the morphism
it follows that it transforms the filtration
$D_{\alpha}=\{x\in Z_p\mid v(x\geq \alpha)$ into $F_{\alpha}$.
So Proposition \ref{type-def-subgroups-of-Qp} finishes the proof.
\end{proof}
From which we get
\begin{prop}\label{case-torus}
If $G$ is an interpretable group and $L\subset G$ is a bounded index type-definable
group with an injective type-definable isomorphism $L\to G(d)$ then
$G$ contains a finite index definable subgroup isomorphic to $F_{\alpha}$
for some $\alpha$. Here $d$ and $F_{\alpha}$ are as in the previous Proposition.
\end{prop}
We end this section by remarking that in the cases analysed the description
of type-definable $p$-adic groups shows in particular that $G^{00}$ is 
the group of infinitesimals. Were the group of infinitesimals
is defined as the kernel of the standard part map 
${\rm st}:G(d)\to G(d)(\mathbb{Q}_p)$ which sends $(x,y)$ to a point
$(x',y')\in\mathbb{Z}_p^{2}$ such that $v(x-x'),v(y-y')>\mathbb{Z}$.
This was already shown in \cite{pillay-onshuus} section 2, in the more
general case of a definably compact 
group definable with parameters in $\mathbb{Q}_p$.
There one sees also that $G^0=G^{00}$. In this generality 
this equality is also seen alternatively
as a consequence of \cite{p-adic-lie} Proposition 26.15.
\section{Subgroups of elliptic curves}\label{elliptic}
In this section we calculate the type-definable subgroups of an elliptic
curve. We will use \cite{elliptic-1} and \cite{elliptic-2} as general references.
We start with a review of properties of elliptic curves.

Here we shall be interested in the $Q_p$-points of an elliptic
curve defined over $Q_p$. For convenience we will take $Q_p$ to be a 
monster model of the analytic language.
We will say definable to mean definable in the valued field language and
$Q_p^{an}$-definable to mean definable in the analytic language.

So assume one is given an equation of the form 
\begin{equation}\label{ell} y^2+a_1xy+a_3y=x^3+a_2x^2+a_4x+a_6\end{equation}
with
$a_1, \cdots, a_6\in Q_p$, such that it has nonzero discriminant
$\Delta\neq 0$, see \cite{elliptic-1} page 42 for the definition of $\Delta$,
it is a polinomial with integer coefficients on the $a_i$. 
Note
also the definition of the $j$-invariant 
as a quotient of another such polinomial and $\Delta$.

Then one takes $E(Q_p)$ to be the set of pairs $(x,y)\in Q_p^2$ that satisfy
this equation and an aditional point $O$ (the point at the infinity).
This is in bijection with the projective closure of the variety defined
by the equation in the plane, this is 
$E(Q_p)\subset \mathbb{P}^2(Q_p)$, 
$[X:Y:Z]\in E(Q_p)$ iff $Y^2Z+a_1XYZ+a_3YZ^2=X^3+a_2X^2Z+a_4XZ^2+a_6Z^3$.
This set $E(Q_p)$ has the estructure of a 
commutative algebraic group, in particular
$E(Q_p)$ forms a group definable in $Q_p$. See III.2.3 of \cite{elliptic-1}
for explicit formulas for the group law.

Now a change of variables $x=u^2x'+r$ and $y=u^3y'+u^2sx'+t$
for $u,r,s,t\in Q_p, u\neq 0$ gives an algebraic group isomorphism
of $E(Q_p)$ and $E'(Q_p)$, where $E'$ is given by coefficients
$a_i'$ given in chapter III Table 3.1 of 
\cite{elliptic-1} ($u^ia'_i$ are polinomials
with integer coefficients on $a_i,s,t,r$). Also included there is the relation
on the discriminants $u^{12}\Delta'=\Delta$.

After a change of variables we can make $a_i\in Z_p$. Among all the equations
obatined by a change of variables there is one such that $a_i\in Z_p$ and 
$v(\Delta)$ is minimal. This is because the set of such $v(\Delta)$ is a 
definable set of $Z$ and so has a minimum. An equation making $v(\Delta)$
minimal is called a minimal Weierstrass equation.

Take now an elliptic curve given by a minimal Weierstrass equation.
Then one can reduce the coefficients of this minimal equation and obtain
a projective algebraic curve over $\mathbb{F}_p$. We take the set
of $\mathbb{F}_p$-points to be 
$\tilde{E}(\mathbb{F}_p)$ the set of pairs $(x,y)\in\mathbb{F}_p^2$
satisfying the reduced Weirstrass equation, together with a point at infinity.
We obtain a map $E(Q_p)\to \tilde{E}(\mathbb{F}_p)$.
If the minimal weirstrass equation used to define $\tilde{E}$ is $f(x,y)$
with discriminant $\Delta$, then if $v(\Delta)=0$ we conclude the 
$\tilde{E}$ is an elliptic curve defined over $\mathbb{F}_p$
and in this case the curve is said to have good reduction.
Otherwise
set $E_{ns}(\mathbb{F}_p)$ the set of pairs $(x,y)\in\mathbb{F}_p^2$ in
$E(\mathbb{F}_p)$ such that
$\frac{\partial}{\partial x}\bar{f}(x,y)\neq 0$
or
$\frac{\partial}{\partial y}\bar{f}(x,y)\neq 0$,
together with the point at infinity.
Then $E(\mathbb{F}_p)\setminus E_{ns}(\mathbb{F}_p)$
consists of a single point $(x_0,y_0)$, 
see Proposition III.1.4 of \cite{elliptic-1}.
Also $E_{ns}(\mathbb{F}_p)$ is group.
If $\bar{f}(x-x_0,y-y_0)=y^2+\bar{a}'_1xy-\bar{a}'_2x^2-x^3$,
then if $d=(\bar{a}'_1)^2+4\bar{a}'_2= 0$, one gets
$E_{ns}(\mathbb{F}_p)\cong (\mathbb{F}_p,+)$ and the curve is said
to have additive reduction.
If $d\neq 0$ and $d$ is not a square, then 
$E_{ns}(\mathbb{F}_p)\cong \{a\in \mathbb{F}_{p^2}^{\times}\mid 
N_{\mathbb{F}_{p^2}/\mathbb{F}_p}(a)=1\}$ is a one dimensional
twisted torus as in Section \ref{torus}, and $E$ is said to have
non-split multiplicative reduction.
If $d\neq 0$ and $d$ is a square then 
$E_{ns}(\mathbb{F}_p)\cong (\mathbb{F}_p^{\times},\times)$,
and $E$ is said to have split multiplicative reduction.
See proposition III.2.5 of \cite{elliptic-1} 
for the definition of this group structure and the isomorphisms indicated
(together with exercise III.3.5 for the nonsplit multiplicative case).
Notice that the set of $a_i$ such that they form a minimal
Weierstrass equation of one of these reduction types, form a definable
set.

One defines in any case $E_0(Q_p)$ to be the the inverse image of
$\tilde{E}_{ns}(\mathbb{F}_p)$ under the reduction map
$E(Q_p)\to \tilde{E}(\mathbb{F}_p)$. Then $E_0(Q_p)$ is a subgroup
of $E(Q_p)$ and $E_0(Q_p)\to \tilde{E}_{ns}(\mathbb{F}_p)$
is a surjective group morphism.
see Proposition VII.2.1 of \cite{elliptic-1}. There this is
proved for $\mathbb{Q}_p$ but the same proof works
in this case also.
One defines also $E_1(Q_p)$ to be the kernel of this map.

We have on $E_1(Q_p)$ a filtration by subgroups
$E_{1,\alpha}(Q_p)$ with $\alpha\in Z_{>0}$ defined by 
$\{(x,y)\in E_1(Q_p)\mid y\neq 0, v(\frac{x}{y})\geq \alpha\}$ together
with the point at infinity. For $p\neq 2$ there is 
an isomorphism in $Q_p^{an}$ from $E_1(Q_p)$
to $(pZ_p,+)$, this isomorphism sends
$D_{\alpha}$ to $E_{1,\alpha}$. For $p=2$ the set
 $E_{1,2}(Q_p)$ has index $p$ in $E_1(Q_p)$ and it is isomorphic
 $Q_p^{an}$ to $(p^2Z_p,+)$, with an isomorphism
wich takes $D_{\alpha}$ to $E_{1,\alpha}$ for $\alpha\geq 2$.
See section IV.1 of \cite{elliptic-1} and Theorem IV.6.4
of \cite{elliptic-1}. We use here that $Q_p$ is elementarily
equivalent to $\mathbb{Q}_p^{an}$ to apply these results to $Q_p$.

As a consequence of Tate's algorithm we have that for
any elliptic curve defined over $\mathbb{Q}_p$
with good, additive or nonsplit multiplicative reduction
the group $E(\mathbb{Q}_p)/E_0(\mathbb{Q}_p)$ is finite of order
$\leq 4$, and if it has split multiplicative reduction
then $v(j(E))<0$,
see section IV.9 in \cite{elliptic-2}.
As $Q_p$ is elementary equivalent to $\mathbb{Q}_p$ we see
that this is also true of $Q_p$.
Finally if $E(Q_p)$ has split multiplicative reduction then 
then there exists an isomorphism 
$O/\langle a\rangle\cong E(Q_p)$ in $Q_p^{an}$
where $O=O(a)$ is as in Section \ref{Qptimes}. See Chapter V.5
in \cite{elliptic-2}.

Now we will give the type-definable subgroups of $H=O/\langle a \rangle$,
up to a subgroup of bounded index. A more complete list can be found in
\cite{thesis}.

We have $O\to H$ the canonical morphism, its restriction to $o(a)$
is injective.
(see Section \ref{Qptimes} for notation).
Then we have the short exact sequence
$1\to o(a)\to H\to \mathbb{R}/\mathbb{Z}\to 1$. 
Let $L\subset H$ be a type-definable subgroup.
Taking the intersection with $o(a)$ we may assume $L\subset o(a)$.
Now $L$ is as in Section \ref{Qptimes}.

This is already enough to obtain an analogue of Proposition 
\ref{case-Qptimes} for the 
language $Q_p^{an}$. We end this section by showing that the groups
obtained in in this Proposition are already definable in $Q_p$ (up to isomorphism).

Take $q\in\mathbb{Q}_p^{\times}$ with $v(q)>0$. Then the uniformization map
$(X,Y):\mathbb{Q}_p^{\times}\to E_q$ of Section V.3 of \cite{elliptic-2}, is
given by $X(u,q)=\frac{u}{(1-u)^2}+\sum_{d\geq 1}(\sum_{m|d}m(u^m+u^{-m}-2)q^d$,
and $Y(u,q)=\frac{u^2}{(1-u)^3}+
\sum_{d\geq 1}(\sum_{m|d}(\frac{1}{2}m(m-1)u^m-\frac{1}{2}m(m+1)u^{-m}+m))q^d$,
for $-v(q)<v(u)<v(q)$ and $u\neq 1$. See page 426 of \cite{elliptic-2}.
If $u=1+t$ for $v(t)>0$, then 
$v(X(u,q))=-2v(t)$ and $V(Y(u,q))=-3v(t)$. Recalling that the uniformizing
map maps $\mathbb{Z}_p^{\times}$ onto $E_0(\mathbb{Q}_p)$ and $U_1$ onto 
$E_1(\mathbb{Q}_p)$, see section V.4 of \cite{elliptic-2}, 
we conclude that it maps $U_{\alpha}$ onto
$E_{1,\alpha}(\mathbb{Q}_p)$. Then the map $O\to E(Q_p)$ does the same.
We return to $\mathbb{Q}_p$.
If $a_d=\sum_{m|d}m(u^m+u^{-m}-2)$ then for $u$ such that
$1\leq v(u)\leq r$ and $2r< v(q)$ we have 
$v(a_d)\geq -dv(u)\geq -dr$ and
$v(a_dq^d)\geq d(v(q)-r)>r\geq v(u)=v(\frac{u}{(1-u)^2})$. So
$v(X(u,q))=v(u)$.
Similarly from the power series for $Y(u,q)$ one gets
$v(Y(u,q))=2v(u)$, for $1\leq v(u)\leq r$ and $3r<q$.
If $r\in\mathbb{Z}$ is such that $1\leq r, 3r<q$, then
the uniformizing map maps $S_r=\{x\in \mathbb{Q}_p^{\times}\mid v(x)=r\}$
into $V_r=\{(x,y)\in E(\mathbb{Q}_p)\mid v(x+y)=v(x)=r<v(y)\}$.
On the other hand we know that $V_r$ is a $E_0(Q_p)$-coset,
see the Lemma V.4.1.4 of \cite{elliptic-2} (and the discussion preceding it).
We conclude that $S_r$ maps onto $V_r$. We see then that
the image $T_r=\{x\in \mathbb{Q}_p^{\times}\mid 1\leq v(x)\leq r\}$
in $E$ is definable (in $r$ and $q$). So the same is true in $Q_p$.
We see then that the image of $v^{-1}[-r,r]\subset O$ in $E$
is definable as a union $u(T_r)\cup E_0\cup u(T_r)^{-1}$.
We conclude that the image of $o(a)\subset O$ is type-definable
and Proposition \ref{localgroupglobal} applies. I will call
the resulting group $O_{E}$, it is a disjoint union of definable sets
with a $Q_p^{an}$-isomorphism $w:O\to O_E$. Call the image
of $o(a)$ in $O_E$ $o_E$, which is type-definable.
Then by compactness
 there is a definable set $S$ which contains $o_E$ and 
which maps injectively into $E$.
By compactness $S$ contains the image of a 
$v^{-1}[-r,r]$ for a $r\in Z$ with 
$r\notin o(v(a))$ and $0\leq r, 3r<v(a)$.
Now by definability of $u(v^{-1}[0,s])$ and $u(U_{\alpha})$ in
$E$ for $0\leq s \leq r$, we conclude that $w(v^{-1}[0,s])$ and
$w(U_{\alpha})$ are definable in $O_E$. Now for any
$s,s'\in v(O)$ the set $v^{-1}[s',s]$ is a finite product
of the sets $v^{-1}[0,t]$ and their inverses, so
$wv^{-1}[s',s]$ is also definable.
We define the group 
$H_E(b)$ for $b\in O_E$ with $vw^{-1}(b)>\mathbb{Z}$ as
having underlying set $wv^{-1}[0,v(b))$ and multiplication
$c\cdot_{H_E(b)}c'=cc'$ if $vw^{-1}(cc')<v(b)$ and 
$cc'b^{-1}$ if $vw^{-1}(cc')\geq v(b)$.
This is a definable group and $w$ induces a $Q_p^{an}$-definable
isomorphism $H(b)\cong H_E(w(b))$. We have also
a definable 
group isomorphism $O_E(b)/\langle b\rangle\cong H_E(b)$.
\begin{prop}\label{case-elliptic}
Let $E$ be an elliptic curve with split multiplicative reduction, 
$G$ an interpretable group and $L\subset G$ a type-definable subgroup
of bounded index. Let
$\phi:L\to E$ be an injective type-definable group morphism.
Then $G$ has a finite index interpretable subgroup $H\subset G$ which is
definably isomorphic to $0$, $E_{1,\alpha}$ or 
$O_E(b)^n/\langle b^n\rangle$.
\end{prop}
\begin{proof}
$\phi(L)$ is a type-definable subgroup of $E$, so after restricting to 
a bounded index type-definable subgroup $\phi(L)\subset o_E(a)$.
 We get a $Q_p^{an}$-type-definable group morphism injection
$\psi:L\to o(a)$, so by the proof of
 Proposition \ref{case-Qptimes} we see that the
restriction of the inverse of $\phi$ to a finite index
type-definable subgroup extends to a definable group morphism
$O(b)^n\to H$ with kernel $\langle b^n\rangle$ or 
$U_{\alpha}\to H$ with trivial kernel,
onto a finite index subgroup $H$ of $G$, and also the composition with $w^{-1}$ is 
definable.
\end{proof}
\section{One dimensional groups definable in $Q_p$}\label{main}
Here we list the one dimensional definable groups, 
up to finite index subgroups and quotient by finite kernel.
This is the main theorem of the document.
\begin{prop}\label{main-teo}
If $G$ is a one dimensional group definable in $Q_p$, then
there exist subgroups $K\subset G'\subset G$ such that
$G'$ is definable of finite index in $G$, $K$ is finite
$G'$ is commutative and $G'/K$ is definably isomorphic to one
of the following groups:
\begin{enumerate}
\item $(Q_p,+)$.
\item $(Z_p,+)$.
\item $((Q_p^{\times})^n,\cdot)$.
\item $(U_{\alpha},\cdot)$. 
Where $U_{\alpha}=\{x\in Q_p^{\times}\mid v(1-x)\geq\alpha\}$.
\item $O(a)^n/\langle a^n\rangle$ as defined in Section \ref{Qptimes}.
\item $F_{\alpha}$. 
Where $F_{\alpha}=\{
\begin{bmatrix} x & dy\\ y & x \end{bmatrix}\mid x^2-dy^2=1, v(1-x)\geq \alpha,
v(y)\geq \alpha\}$
and $d\in \mathbb{Q}_p^{\times}\setminus (\mathbb{Q}_p^{\times})^2$,
and $v(d)\geq 0$.
\item $E_{1, \alpha}$. For $E$ an elliptic curve.
\item $O_E(b)^n/\langle b^n\rangle$. For $E$ a Tate Elliptic curve.
\end{enumerate}
The definition of $E_{1, \alpha}$ and $O_E(b)$ are 
in Section \ref{elliptic}.
\end{prop}
\begin{proof}
By Proposition \ref{abelian-by-finite}, we may assume
$G$ is commutative, and Theorem \ref{mop-teo} applies.
So we get a type-definable bounded index group
$T\subset G$ and an algebraic group $H$ and a 
type-definable group morphism $T\to H$ with finite kernel.
Replacing $H$ by the Zariski closure of the image of $T\to H$ we may assume
$H$ is a one-dimensional algebraic group.
Replacing $H$ by the connected component of the identity
(which is a finite index algebraic subgroup) we assume
that $H$ is a connected algebraic group.
Then $H$ is isomorphic as an algebraic group
to the additive group or the multiplicative group or
the one dimensional twisted torus or an elliptic curve.
These cases are dealt with in Propositions \ref{generic-to-subgroup-of-Qp}
\ref{case-Qptimes}, \ref{case-torus}, Section \ref{elliptic},
and Proposition \ref{case-elliptic} for the Tate curve case.
\end{proof}
We remark that Proposition \ref{finite-torsion} below implies that the finite
kernel is unnecessary in all cases except maybe the lattice ones
5) and 8). In these cases I do not know if it is necessary.
\begin{prop}\label{universally-injective-implies-split}
Suppose $K\subset G$ are abelian groups and consider the conditions:
\begin{enumerate}
\item $G$ is a Hausdorff topological group and $K$ is compact.
\item $nG\cap K=nK$.
\end{enumerate}
If 1) occurs, 
then the injection of abstract groups $K\to G$ splits iff condition 2) occurs.

If $K$ is finite then 1) is true for the discrete topology in $G$.
If $G/K$ is torsion free then 2) is true.
\end{prop}
\begin{proof}
Assume 1). 
Suppose that $0\to K\to G\to G/K\to 0$ splits. Then tensoring by
$\mathbb{Z}/n\mathbb{Z}$ remains exact, which is exactly 2).

Now assume 1) and 2).
Choose a set theoretic section $\phi:G/K\to G$. Then the set
of all such sections is in biyective correspondence with the maps
$K^{G/K}$, and the set of group sections is closed in the product topology.
As $K^{G/K}$ is a compact topological space it is then enough
to show that for all finitely generated subgroups $A\subset G/K$
the map $\pi^{-1}A\to A$ splits, that is, without loss of generality
$G/K$ is finitely generated.
Take $\mathbb{Z}^n\to G/K$ is a surjective group homomorphism
and $\alpha:\mathbb{Z}^n\to G$ is a lift.
Then after a base change one may assume that the kernel of $\mathbb{Z}^n\to G/K$
is $T=n_1\mathbb{Z}\times\cdots\times n_r\mathbb{Z}$. 
From the assumption we conclude
that ther exists $\beta:\mathbb{Z}^n\to K$ such that $\alpha-\beta$ has kernel
$T$. That is, $\alpha-\beta$ factors as a section $G/K\to G$ as required.
\end{proof}
We note that condition 2) is equivalent to universal injectivity of the
map of $\mathbb{Z}$-modules $K\to G$, and replacing this condition
for universal injectivity of topological $R$-modules it remains
true that it is equivalent to splitting. 
Similarly $G/K$ is torsion-free iff it is 
flat as a $\mathbb{Z}$-module, and this implies condition 2) in the setting
of $R$-modules too. 

We note also that if $G$ is an invariant 
abelian group (or $R$-module) 
with relatively type-definable product 
and $K$ is a type-definable
subgroup of bounded index the proposition is also true using logic
compactness in the proof.
\begin{prop}\label{finite-torsion}
Take $L$ a definable abelian group in some language.
Assume $A\subset L$ is the torsion part of $L$ and is injective.
Assume that for every $n$ $[L:nL]$ is finite.
Then $A\to L$ is split injective and any retraction $L\to A$ is definable.
\end{prop}
\begin{proof}
As $A$ is the torsion part, $L/A$ is torsion free. So
by Proposition \ref{universally-injective-implies-split}
we obtain that $A\to L$ is split injective.
If $L\to A$ is a retraction then it factors throug
$nL$ for $n={\rm Card}A$, so it is definable.
\end{proof}

\end{document}